\documentclass[12pt,a4paper]{article} 
\usepackage{amsmath}
\usepackage{amsfonts}
\usepackage{amssymb}
\usepackage{cite}
\usepackage{graphics}
\usepackage{amssymb,amscd}
\usepackage[utf8]{inputenc}
\usepackage[T2A]{fontenc}
\usepackage{tikz-cd}
\usepackage[russian,english]{babel}

\usepackage{amscd}
\usepackage{yfonts}
\usepackage[normalem]{ulem}
\usepackage{esint}
\usepackage{amsmath}
\usepackage{amsfonts}
\usepackage{amssymb}
\usepackage{hyperref}
\usepackage{verbatim}

\usepackage{graphics}
\usepackage{amsfonts,amssymb,amsmath,mathrsfs,amsthm,amscd}
\usepackage{tikz} 
\usepackage{tikz-cd} 
\usetikzlibrary{decorations.pathreplacing}

\theoremstyle{plain}

\newtheorem{theo}{Theorem}[section]
\newtheorem{prop}[theo]{Proposition}

\newtheorem{lemm}[theo]{Lemma}
\newtheorem{coro}[theo]{Corollary}

\theoremstyle{definition}
\newtheorem{defi}{Definition}

\theoremstyle{remark}
\newtheorem*{rema}{Remark}

\numberwithin{equation}{section}

\DeclareMathOperator{\tr}{tr}

\DeclareMathOperator{\Res}{Res}
\DeclareMathOperator{\Spec}{spec}
\DeclareMathOperator{\Scal}{Scal}
\DeclareMathOperator{\Ric}{Ric}
\DeclareMathOperator{\R}{R}
\DeclareMathOperator{\Vol}{vol}
\DeclareMathOperator{\dvol}{dvol}
\DeclareMathOperator{\genus}{genus}

\newcommand{\stkout}[1]{\ifmmode\text{\sout{\ensuremath{#1}}}\else\sout{#1}\fi} 

\title{On the spectral geometry of manifolds with conic singularities.}
\author{Asilya Suleymanova}
\date{October 2017}

\begin{document}
\maketitle

\begin{abstract}
In the previous article we derived a detailed asymptotic expansion of the heat trace for the Laplace-Beltrami operator on functions on manifolds with conic singularities. In this article we investigate how the terms in the expansion reflect the geometry of the manifold. Since the general expansion contains a logarithmic term, its vanishing is a necessary condition for smoothness of the manifold. In the two-dimensional case this implies that the constant term of the expansion contains a non-local term that determines the length of the (circular) cross section and vanishes precisely if this length equals $2\pi$, that is, in the smooth case. We proceed to the study of higher dimensions. In the four-dimensional case, the logarithmic term in the expansion vanishes precisely when the cross section is a spherical space form, and we expect that the vanishing of a further singular term will imply again smoothness, but this is not yet clear beyond the case of cyclic space forms.
In higher dimensions the situation is naturally more difficult. We illustrate this in the case of cross sections with constant curvature. Then the logarithmic term becomes a polynomial in the curvature with roots that are different from 1, which necessitates more vanishing of other terms, not isolated so far.
\end{abstract}

\tableofcontents

\begin{section}{Introduction}

For compact Riemannian manifolds $(M,g)$, the problem of finding geometric information from the eigenvalues of the Laplace-Beltrami operator and the Hodge Laplacian has been extensively studied, see e.g. \cite{G} and the references given there. On closed $(M,g)$ there is an asymptotic expansion
\begin{equation}\label{smooth expansion}
\tr e^{-t\Delta}\sim_{t\to+0}(4\pi t)^{-\frac{m}{2}}\sum_{j=0}^{\infty} a_jt^j,
\end{equation}
where $a_j\in\mathbb{R}$.
In principle, every term in (\ref{smooth expansion}) can be written as an integral over the manifold of a local quantity. Namely, 
\begin{align}\label{smooth terms}
a_j=\int_M u_j\dvol_M,
\end{align}
where $u_j$ is a polynomial in the curvature tensor and its covariant derivatives. In particular, $u_0=1$ and $u_1=\frac16\Scal$, where $\Scal$ is the scalar curvature of $(M,g)$. The bigger $j$, the more complicated the calculation of $u_j$. Sometimes we write $u_j(p)$ to indicate that it is a local quantity, i.e.~it depends on a point $p\in M$.

There are many examples of manifolds that are {\itshape isospectral}, i.e.~have the same spectrum of $\Delta$, but are not isometric, see the survey \cite{GPS}. However, it remains very interesting to study to what extent the geometry of $(M,g)$ can be determined from $\Spec\Delta$.

In this article we study spectral geometry of a non-complete smooth Riemannian manifold $(M,g)$ that possesses a conic singularity. By this we mean that there is an open subset $U$ such that $M\setminus U$ is a smooth compact manifold with boundary $N$. Furthermore, $U$ is isometric to $(0,\varepsilon)\times N$ with $\varepsilon>0$, where the {\itshape cross-section} $(N,g_N)$ is a closed smooth manifold, and the metric on $(0,\varepsilon)\times N$ is
\begin{align}\label{conic metric}
g_{\text{conic}}=dr^2+r^2g_{N}, \;\;\; r\in(0,\varepsilon).
\end{align}

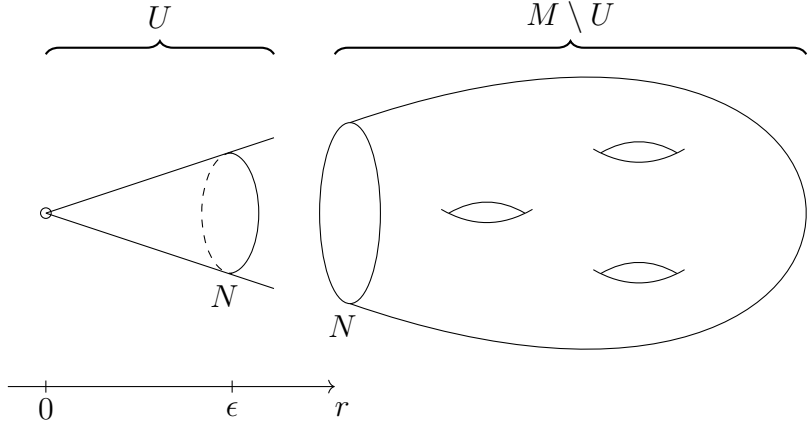
\begin{figure}[!ht] 
\centering
\begin{tikzpicture}[scale=1]
    \draw (0,0) -- (3,1);
    \draw (0,0) -- (3,-1);
    \draw (0,0) circle [radius=0.07];
    \draw[dashed] (2.45,0.8) arc (90:270:0.4cm and 0.8cm);
    \draw (2.4,-0.8) arc (-90:90:0.4cm and 0.8cm);
    \draw (2.35,-1.1) node {$N$};
    \draw[decorate,decoration={brace,raise=3pt,amplitude=5pt}, thick] (0,2)-- (3,2);
    \draw (1.5,2.6) node {$U$};

    \draw (4,0) circle [x radius=0.4cm, y radius=1.2cm];
    \draw (4,1.2) .. controls (12,4) and (12,-4) .. (4,-1.2);
    \draw (3.9,-1.5) node {$N$};
    \draw[decorate,decoration={brace,raise=3pt,amplitude=5pt}, thick] (3.8,2)-- (10,2);
    \draw (6.9,2.6) node {$M \setminus U$};

    \draw (5.3,0) to[bend left] (6.3,0);
    \draw (5.2,.05) to[bend right] (6.4,.05);

    \draw (7.3,0.8) to[bend left] (8.3,0.8);
    \draw (7.2,0.85) to[bend right] (8.4,0.85);

    \draw (7.3,-0.8) to[bend left] (8.3,-0.8);
    \draw (7.2,-0.75) to[bend right] (8.4,-0.75);

    \draw[->] (-0.5,-2.3) -- (3.8,-2.3);
    \draw (0,-2.38) -- (0,-2.22); 
    \draw (2.45,-2.38) -- (2.45,-2.22);
    \draw (0,-2.6) node {$0$}; 
    \draw (2.45,-2.6) node {$\epsilon$};
    \draw (3.9,-2.6) node {$r$}; 

\end{tikzpicture}
\caption{Manifold with a conic singularity.}
\end{figure}

For a particular choice of $g_N$, the conic metric $g_{\text{conic}}$ provides an isometry with the punctured ball. Namely, let $(N, g_N)$ be a unit sphere with the round metric, then $U$ is isometric to a punctured ball with the metric $g_{\text{conic}}$. If we include the conic point to the neighbourhood $\bar{U}:=[0,\varepsilon)\times N$, we see that there is no singularity, but rather we have polar coordinates in the neighbourhood $\bar{U}$. Informally speaking a conic singularity is not necessarily a singularity, it will then be referred to as the apparent singularity. We illustrate this in the two-dimensional case on Figure~2. The metric on $U$ is $g_{\text{conic}}=dr^2+r^2\sin^2\alpha d\theta^2$, where the parameter $0<\alpha\leq\pi/2$ denotes the angle between the generating line and the axis of the cone and $0<\theta\leq2\pi$ is the coordinate on $S^1$.

\begin{figure}[!ht] 
\centering
\begin{tikzpicture}[scale=1]

    \draw (0,0) -- (3,1);
    \draw (0,0) -- (3,-1);
    \draw (0,0) -- (2.45,0);
    \draw (0,0) circle [radius=0.07];
    \draw[dashed] (2.45,0.8) arc (90:270:0.4cm and 0.8cm);
    \draw (2.4,-0.8) arc (-90:90:0.4cm and 0.8cm);

    \draw (0,-1) node {$0 < \alpha < \frac{\pi}{2}$};

    \draw (0.6,0) .. controls (0.65,0.1) .. (0.6,0.2);
    \draw (0.8,0.12) node {\scalebox{0.7}{$\alpha$}};

    \draw[dashed] (8,0) circle [radius=1cm];
    \draw (8,0) circle [radius=0.07];
    \draw (7,-1.2) node {$\alpha=\frac{\pi}{2}$};
\end{tikzpicture}
\caption{Actual singularity vs apparent singularity.}
\end{figure}
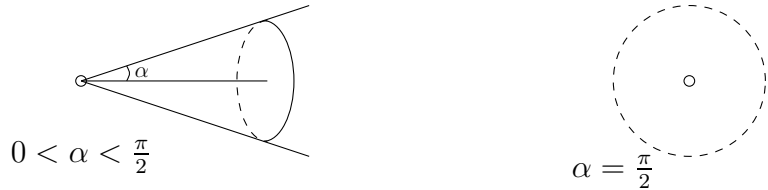

\begin{theo}\label{main theorem}
Let $\Delta$ be the Laplace-Beltrami operator on smooth functions with compact support on $(M,g)$. If $m\geq4$, then $\Delta$ is essentially self-adjoint operator, otherwise we consider the Friedrichs extension of $\Delta$. Denote the self-adjoint extension of the Laplace-Beltrami operator by the same symbol~$\Delta$. Then
\begin{equation}\label{main expansion}
\tr e^{-t\Delta}\sim_{t\to0+}(4\pi t)^{-\frac{m}{2}}
\sum_{j=0}^{\infty}\tilde{a}_jt^j
+b
+c\log t,
\end{equation}
\begin{itemize}
\item[(a)]where
\begin{align*}
\tilde{a}_j=
\begin{cases}
\int_Mu_j\dvol_M \text{ for } j\leq m/2-1, \\
\fint_Mu_j\dvol_M \text{ for } j>m/2-1.
\end{cases}
\end{align*}
Above $\fint$ denotes the regularized integral, which is defined in \cite[Section~2.1]{L}, of local quantities $u_j$ in (\ref{smooth terms}).
\item[(b)] The constant term $b$ in general cannot be written in terms of local quantities, and is given by
\begin{align*}
b
=&-\frac{1}{2}\Res_0\zeta^{\frac{m-2}{2}}_N(-1/2)
+\frac{\Gamma'(-\frac1 2)}{4\sqrt{\pi}}\Res_1\zeta^{\frac{m-2}{2}}_N(-1/2)\\
&-\frac1 4\sum_{1\leq j\leq\frac{m}{2}}j^{-1}B_{2j}\Res_1\zeta^{\frac{m-2}{2}}_N(j-1/2),
\end{align*}
where $\zeta^l_N(s)=\sum_{\lambda\in\Spec\Delta_N}(\lambda+l^2)^{-s}$ is the spectral zeta function shifted by $l$. The constants $B_{2j}$ are the Bernoulli numbers, $\Res_0f(s_0)$ is the regular analytic continuation of a function $f(s)$ at $s=s_0$, and $\Res_1f(s_0)$ is the residue of the function $f(s)$ at $s=s_0$.
\item[(c)] The logarithmic term is given by
\begin{equation}\label{logarithmic term}
\begin{split}
c&=\frac12Res_1\zeta^{\frac{n-1}{2}}_N(-1/2)\\
&=
\begin{cases}
\frac{1}{2(4\pi)^{\frac{m}{2}}}\sum_{k=0}^{\frac{m}{2}}(-1)^{k+1}\frac{(m-2)^{2k}}{4^k k!}a^N_{\frac{m}{2}-k}, &\text{ for }m \text{ -- even},\\
0, &\text{ for }m \text{ -- odd}.
\end{cases}
\end{split}
\end{equation}
\item[(d)] If $c=0$ then $\tilde{a}_{m/2}=\int_Mu_{m/2}\dvol_M$ does not have a contribution from the singularity.
\end{itemize}
Here $\Delta_N$ is the Laplace-Beltrami operator on the cross-section $(N,g_N)$ and $a^N_j, j\geq0$ denote the coefficients in the heat trace expansion (\ref{smooth expansion}) on $(N,g_N)$.
\end{theo}

Above we need to regularize the integrals, because in general $\int_Mu_j\dvol_M$ diverges. If for some $j\geq0$ the integral converges, i.e. $\fint_Mu_j\dvol_M=\int_Mu_j\dvol_M$, then in this case $\tilde{a}_j$ is equal to $a_j$ from (\ref{smooth terms}).

Theorem~\ref{main theorem} allows to connect the coefficients in (\ref{main expansion}) to the geometry of $(M,g)$. It is now natural to pose the following question: given the coefficients in (\ref{main expansion}), can we say if there are actual or only apparent singularities? The idea is to compare the expansion (\ref{main expansion}) to the expansion on a smooth compact manifold (\ref{smooth expansion}). In case $c\neq0$ we have an actual singularity, whereas if $c=0$ we need to compute $b$ to detect a singularity.

Consider now the case of even-dimensional manifolds. While the logarithmic term $c$ is written in terms of the geometry of the cross-section near the singularity, the constant term $b$ is expressed in residues and regular values of the spectral zeta function of the cross-section. Thus it is difficult to extract the geometric meaning of $b$ in general; therefore, we study low dimensions one by one.

The heat trace expansion for the Friedrichs extension of the Laplace-Beltrami operator on the algebraic curves was developed by Br\"uning and Lesch in \cite[Theorem~1.2]{BL}. In general, many logarithmic terms $c_{j}t^j\log t$ with $j\geq0$ might appear in the expansion. We prove that in the heat trace expansion on a surface with conic singularities there is no logarithmic term and the geometric information about the singularities is encoded only in the constant term. In this case the spectral zeta function on the cross-section $(N,g_N)$ is the Riemann zeta function, so the computations can be done explicitly.

\begin{lemm}\label{lemma expansion on surface}
Let $(M,g)$ be a surface with $l$ conic singularities. Let $\Delta$ be the Friedrichs extension of the Laplace-Beltrami operator. Then
\begin{equation}
\tr e^{-t\Delta}\sim_{t\to0+}\frac{1}{4\pi t}
\sum_{j=0}^{\infty}a_jt^j
+\frac{1}{12}\sum_{i=1}^l\left(\frac{1}{\sin\alpha_i}-\sin\alpha_i\right),
\end{equation}
where $\alpha_i$ is the angle between the generating line and the axis of the cone corresponding to the $i$-th conic singularity. Above $a_j$ does not have any contribution from the singularities for $j\geq0$.
\end{lemm}

Denote by $(\bar M,g)$ the complete surface with conic points included, i.e.~the neighbourhoods of conic points are $U'=[0,\varepsilon)\times N$. By the expansion in Lemma~\ref{lemma expansion on surface}, we obtain the following.

\begin{theo}\label{isospectral surfaces}
Let $(\bar M,g)$ be a complete simply-connected surface with conic singularities. If $(\bar M,g)$ has at least one singularity, then it is not isospectral to any smooth closed surface.
\end{theo}

In even dimensions $m\geq4$ the cross-section $(N,g_N)$ is a closed manifold of dimension three or more. The spectral zeta function of $N$ is known explicitly only in very few cases which makes the situation in dimensions $m\geq4$ much more complicated. We present the spectral geometry of a four-dimensional $(M,g)$.

\begin{theo}\label{4 dim theorem}
Let $(M,g)$ be a four-dimensional manifold with conic singularities.
\begin{itemize}
\item[(1)] The logarithmic term in the heat trace expansion (\ref{main expansion}) is equal to zero if and only if the cross-section of every singularity is isometric to a spherical space form.
\item[(2)] Assume that the logarithmic term in the heat trace expansion (\ref{main expansion}) is equal to zero. Then $\tilde{a}_j=a_j$ for $j\geq0$. In this case only the constant term in the heat trace expansion has a contribution from the singularities.
\end{itemize}
\end{theo}

Denote by $(\bar M,g)$ the manifold with conic singularities with conic points included. 

\begin{coro}
Let $(\bar M,g)$ be a complete four-dimensional manifold with conic singularities. If $(\bar M,g)$ has at least one singularity with a cross-section $(N,g_N)$ not isometric to a spherical space form, then it is not isospectral to any smooth compact four-dimensional manifold.
\end{coro}

At some point there was a hope that a theorem similar to Theorem~\ref{4 dim theorem} holds true for any even-dimensional manifold with conic singularities, so we determine a criterion for the logarithmic term in the heat trace expansion to vanish.

\begin{lemm}\label{logarithmic term is zero}
Let $(M,g)$ be a even-dimensional manifold with a conic singularity. The logarithmic term in the heat trace expansion on $(M,g)$ is equal to zero if and only if the following equality holds for the heat trace coefficients of the $n$-dimensional cross-section manifold $(N,g_N)$
\begin{align}
a^N_{\frac{n+1}{2}}=\sum_{k=1}^{\frac{n+1}{2}}(-1)^{k+1}\frac{(n-1)^{2k}}{4^k k!}a^N_{\frac{n+1}{2}-k}.
\end{align}
\end{lemm}

\begin{rema}
In the $n=1$ case this equality is always true. In the $n=3$ case this equality implies that the sectional curvature of the cross-section manifold $(N,g_N)$ is constant $\kappa=1$ (by Theorem~\ref{4 dim theorem}). Let us analyse what geometric restrictions we obtain from this equality in higher dimensional cases.
\end{rema}

\begin{theo}\label{constant curvature theorem}
Assume that the cross-section manifold $(N,g_N)$ has constant sectional curvature $\kappa$. Then the logarithmic term in the heat trace expansion on $(M,g)$ can be written as the polynomial in $\kappa$ of degree $\frac{n+1}{2}$
$$
c=
\frac{1}{4\sqrt{\pi}}\frac{\Vol(N)}{\Vol(S^n)}\sum_{k=0}^{\frac{n+1}{2}}(-1)^{k+1}\frac{(n-1)^{2k}}{4^kk!}\sum_{l=1}^{\frac{n-1}{2}}\frac{(\frac{n-1}{2})^{2l-2k+2}\Gamma(l+\frac12)K_l^{\frac{n-1}{2}}}{(l-k+1)!(n-1)!}\kappa^{\frac{n+1}{2}-k},
$$
where numbers $K_l^{\frac{n-1}{2}}$ are given by (\ref{poly}) and depend only on $n$ and $l$.
\end{theo}

The next results show that Theorem~\ref{4 dim theorem} cannot be extended to higher dimensions.

\begin{coro}\label{6 dim theorem}
Let $(N,g_N)$ be a five-dimensional manifold with constant sectional curvature $\kappa$. The logarithmic term in the heat trace expansion on $(M,g)$ is zero if and only if $\kappa=1$ or $\kappa=2$.
\end{coro}

\begin{coro}\label{8 dim theorem}
Let $(N,g_N)$ be a seven-dimensional manifold with constant sectional curvature $\kappa$. The logarithmic term in the heat trace expansion on $(M,g)$ is zero if and only if $\kappa=1$ or $\kappa=\frac{225}{109}\pm\frac{36\sqrt{5}}{109}$.
\end{coro}

From the above results we conclude that the higher the dimension of the manifold $(M,g)$, the less geometric information we can obtain from the heat trace expansion on $(M,g)$. If $(\bar M,g)$ is a complete simply-connected surface with conic singularities, from the heat trace expansion on $(M,g)$ we can determine whether $(M,g)$ has an actual singularity or an apparent singularity. If $(M,g)$ is a four-dimensional manifold with conic singularities, we can determine whether a cross-section of the singularity is isometric to a spherical space form. If $(M,g)$ is a higher dimensional manifold, the situation becomes less determined.

{}
This article is organized as follows. In Section~\ref{results}, we prove Lemma~\ref{lemma expansion on surface}, Theorem~\ref{isospectral surfaces}, Theorem~\ref{4 dim theorem} and Lemma~\ref{logarithmic term is zero}. In Section~5, we prove Theorem~\ref{constant curvature theorem}, Corollary~\ref{6 dim theorem} and Corollary~\ref{8 dim theorem}. We conclude with the computations of the logarithmic term and the constant term for some particular $n$-dimensional cross-sections.

\end{section}

\begin{section}{Geometrical information from the heat trace expansion}\label{results}

In this section we analyse the coefficients in the expansion in Theorem~\ref{main theorem} from a geometrical point of view. First, we consider the two-dimensional case and show in which setting we can distinguish a surface with conic singularities from a compact smooth surface using the heat trace expansion.

Next, we consider the four-dimensional case and develop a criterion of the cross-section $(N,g_N)$ near a singularity to be isometric to a spherical space form. Then we study out some illustrative four-dimensional examples.

We complete the section with a geometrical criterion for the logarithmic term in the heat trace expansion in Theorem~\ref{main theorem} to be equal to zero that holds for any even-dimensional $(M,g)$. Recall that for any odd-dimensional $(M,g)$ the logarithmic term is identically zero.

\begin{subsection}{Compact surfaces with conic singularities}

The constant term in the heat trace expansion on surfaces with conic singularities was computed in \cite[pp.423--424]{BS2}. Here we prove that the logarithmic term in the expansion is zero and $\tilde{a}_j$, in Theorem~\ref{main theorem}, does not need regularization for any $j\geq0$.

Consider a surface $(M,g)$ with $l$ conic singularities. By this we mean $(M,g)$, which possesses an open set $\cup_{i=1}^{l}U_i$ such that $M\setminus\cup_{i=1}^{l}U_i$ is a compact surface with boundary $N_1\cup\dots\cup N_l$. Furthermore, $U_i$ is isometric to $(0,\varepsilon)\times N_i$ with $\varepsilon>0$, where $(N_i,g_{N_i})$ is a closed one-dimensional manifold for $1\leq i\leq l$. The metric in the neighbourhood $U_i$ is
$$
g_{U_i}=dr^2+r^2g_{N_i},
$$
where $g_{N_i}:=\sin^2\alpha_i d\theta$. Here $0<\theta\leq2\pi$ is the local coordinate on $S^1$ and $0<\alpha_i\leq\pi/2$ is the angle between the generating line and the axis of the cone $U_i$.

\begin{lemm}\label{lemma expansion on surface here}
The heat trace expansion on a surface $(M,g)$ with $l$ conic singularities has the following form
\begin{equation}
\tr e^{-t\Delta}\sim_{t\to0+}\frac{1}{4\pi t}
\sum_{j=0}^{\infty}a_jt^j
+\frac{1}{12}\sum_{i=1}^l\left(\frac{1}{\sin\alpha_i}-\sin\alpha_i\right),
\end{equation}
where $a_j$ does not have any contribution from the singularities for $j\geq0$.
\end{lemm}

\begin{proof}
By Theorem \ref{main theorem}, the heat trace expansion on $M$ is
\begin{align}\label{expansion on surface}
\tr e^{-t\Delta}\sim_{t\to0+}\frac{1}{4\pi t}
\sum_{j=0}^{\infty}\tilde{a}_jt^j
+b
+c\log t,
\end{align}
where $\tilde{a}_j$ are integrals over $(M,g)$, which might need regularization.

First, we show that $c=0$. Observe that every singularity contributes to the logarithmic term. The scalar curvature of a circle is zero, hence by Theorem~\ref{main theorem}~(c),
$$
c=\sum_{i=1}^l c_i=-\frac{1}{2(4\pi)}\sum_{i=1}^la^{N_i}_1=-\frac{1}{16\cdot6\pi}\sum_{i=1}^l\int_{N_i}\Scal\dvol_{N_i}=0.
$$

Next, we compute the constant term in the heat trace expansion (\ref{expansion on surface}). Since $\text{Spec }\Delta_{N_i}=\{\sin\alpha_i^{-2}k^2\;|\;k\in\mathbb{Z}\}$, the spectral zeta function on $N_i$ is
$$
\zeta^0_{N_i}(s)=\sum_{k\in\mathbb{Z}\setminus\{0\}}\left(\frac{k}{\sin\alpha_i}\right)^{-2s}=2\sin\alpha_i^{2s}\zeta(2s),
$$
where $\zeta(s)$ is the Riemann zeta function.

By the regular analytic continuation of the Riemann zeta function, we obtain
$$
\Res_0\zeta^0_{N_i}(-1/2)=-2\sin\alpha_i^{-1}\frac{1}{12}=-\frac{1}{6\sin\alpha_i}.
$$
The residues are
$$
\Res_1\zeta^0_{N_i}(j-1/2)=\Res\zeta(2j-1)=0,
\text{ for } j\geq2
$$
and
$$
\Res_1\zeta^0_{N_i}(1/2)=2\sin\alpha_i\Res\zeta(1)=2\sin\alpha_i.
$$
Thus, the constant term in the heat trace expansion is
$$
b=\sum_{i=1}^l b_i
=\sum_{i=1}^l\left(\frac{1}{12\sin\alpha_i}-\frac14 B_2 2\sin\alpha_i\right).
$$
Hence
$$
b=\frac{1}{12}\sum_{i=1}^l\left(\frac{1}{\sin\alpha_i}-\sin\alpha_i\right).
$$

Next, we compute coefficients $\tilde{a}_j$ in (\ref{expansion on surface}). By Theorem~\ref{main theorem}~(a),
$$
\tilde{a}_j
=\fint_M u_j(p)\dvol_M
=\int_M\psi u_j(p)\dvol_M+\sum_{1\leq i\leq l}\fint_M\varphi_i(r)u_j(p_i)\dvol_M,
$$
where $\psi, \varphi_1(r),\dots \varphi_l(r)$ is a partition of unity on $(M,g)$ such that $\psi$ has support away from the singular points and $\varphi_i$ has support in $(U_i,g_{U_i})$.

The curvature tensor on the cone over a circle is equal to zero, hence  we have $u_j(p)\equiv0$ for $p\in U_i$ for $1\leq i\leq l$ and $j>0$, also $u_0(p)\equiv1$ for $p\in M$. Therefore, the singularities do not contribute to the positive power terms in the expansion and
\begin{equation*}
\tilde{a}_j=
\begin{cases}
\Vol(M) \text{ for }i=0,\\
\int_Mu_j(p)\dvol_M, \text{ for }j>0.
\end{cases}
\end{equation*}
Above neither $\tilde{a}_j$ needs regularization, i.e. $\tilde{a}_j$ is equal to $a_j$ in (\ref{smooth expansion}). We conclude that the geometric information about the singularities is encoded only in the constant term. Finally,
\begin{equation}
\tr e^{-t\Delta}\sim_{t\to0+}\frac{1}{4\pi t}
\sum_{j=0}^{\infty}a_jt^j
+\frac{1}{12}\sum_{i=1}^l\left(\frac{1}{\sin\alpha_i}-\sin\alpha_i\right).
\end{equation}
\end{proof}

\begin{defi}
Two Riemannian manifolds $(M_1,g_1)$ and $(M_2,g_2)$ are called {\itshape isospectral} if the eigenvalues of their Laplace-Beltrami operators coincide.
\end{defi}

Denote by $(\bar{M},g)$ the complete surface with conic singularities such that the conic points belong to $(\bar{M},g)$, i.e.~the neighbourhood of a conic point, $U'_i$, is isometric to $[0,\varepsilon)\times N_i$ for $1\leq i\leq l$.

\begin{theo}
Let $(\bar{M},g)$ be a complete simply-connected surface with conic singularities. If $(\bar{M},g)$ has at least one singularity, then it is not isospectral to any smooth closed surface.
\end{theo}
\begin{proof}
By Lemma \ref{lemma expansion on surface here},
\begin{align*}
\tr e^{-t\Delta}\sim_{t\to0+}
&\frac{1}{24\pi}\int_M\tilde\Scal(p)\dvol_M\\
&+\frac{1}{12}\sum_{i=1}^l\left(\frac{1}{\sin\alpha_i}-\sin\alpha_i\right)+\frac{1}{4\pi t}
\sum_{j\geq0,j\neq1}a_jt^j,
\end{align*}
where $\tilde\Scal(p)$ is the scalar curvature at $p\in M$.

By \cite[pp.52-53]{BL}, 
$$
\frac12\int_M\tilde\Scal(p)\dvol_M=2\pi(2-2\genus(\bar M)).
$$
Since $(\bar{M},g)$ is simply-connected, $\genus(\bar M)=0$ and

\begin{align*}
\tr e^{-t\Delta}\sim_{t\to0+}
\frac13+\frac{1}{12}\sum_{i=1}^l\left(\frac{1}{\sin\alpha_i}-\sin\alpha_i\right)+\frac{1}{4\pi t}
\sum_{j\geq0,j\neq1}a_jt^j.
\end{align*}
Observe that
$$
\frac{1}{\sin\alpha_i}-\sin\alpha_i\geq0,
$$
and the equality holds if and only if $\alpha_i=\frac{\pi}{2}$, i.e.~the corresponding neighbourhood $U'_i$ of the singularity is a disk. Therefore, if $(\bar{M},g)$ has at least one singularity, we obtain the equality for the constant coefficient in the heat trace expansion
$$
\frac13+\frac{1}{12}\sum_{i=1}^l\left(\frac{1}{\sin\alpha_i}-\sin\alpha_i\right)>\frac13.
$$

Let $(M_1,g_1)$ be a smooth closed surface and $\Delta_{M_1}$ be the Laplace-Beltrami operator on it. Then by (\ref{smooth expansion}), the heat trace expansion on $(M_1,g_1)$ has the following form
\begin{align*}
\tr e^{-t\Delta_{M_1}}\sim_{t\to0+}
\frac13(1-\genus(M_1))+\frac{1}{4\pi t}
\sum_{j\geq0,j\neq1}a'_jt^j.
\end{align*}
The constant coefficient in the heat trace expansion satisfies the following equality
$$
\frac13(1-\genus(M_1))\leq\frac13.
$$

We conclude that since the constant term in the heat trace expansion on $(\bar{M},g)$ with at least one conic singularity is grater than $\frac13$, $(\bar{M},g)$ and $(M_1,g_1)$ cannot be isospectral.
\end{proof}

\end{subsection}

\begin{subsection}{Four-dimensional manifolds with conic singularities}\label{4 dim}

In this section we consider the spectral geometry of four-dimensional manifolds with conic singularities. As before, we use symbols with tildes to denote tensors on $(M,g)$, and without tildes to denote tensors on the cross-section $(N,g_N)$.
\begin{theo}
Let $(M,g)$ be a four-dimensional manifold with conic singularities.
\begin{itemize}
\item[(1)] The logarithmic term in the heat trace expansion on $(M,g)$  is equal to zero if and only if the cross-section of every singularity is isometric to a spherical space form.
\item[(2)] Assume that the logarithmic term in the heat trace expansion on $(M,g)$  is equal to zero. Then no regularization is needed for $\tilde{a}_j$, for any $j\geq0$. In this case only the constant term in the heat trace expansion has a contribution from the singularities.
\end{itemize}
\end{theo}

\begin{proof}

We consider one conic singularity and show that its contribution to the logarithmic term is non-positive. Moreover the contribution is equal to zero if and only if the cross-section is isometric to a spherical space form. Consequently, if $(M,g)$ has more conic singularities, the logarithmic term is equal to zero if and only if each cross-section is isometric to a spherical space form. Similarly, we show that the second claim of the theorem holds for a manifold with one conic singularity, i.e. we do not need to regularise integrals in $\tilde{a}_j$, hence it holds for more singularities.

We denote the heat trace coefficients in (\ref{smooth expansion}) on $(N,g_N)$ by $a^N_j$, $j\geq0$. By Theorem~\ref{main theorem}~(c),
\begin{align*}
c
=\frac{1}{4(4\pi)^2}\sum_{k=0}^2(-1)^{k+1}\frac{1}{k!}a^N_{2-k}
=\frac{1}{4(4\pi)^2}(-a^N_2+a^N_1-\frac12 a^N_0).
\end{align*}

By \cite[p.201 Theorem 3.3.1]{G}
\begin{equation}\label{term u_2}
u_2(p)=\frac{1}{360}\left(12\Delta \Scal(p)+5\Scal(p)^2-2\lvert\Ric(p)\rvert^2+2\lvert\R(p)\rvert^2\right),
\end{equation}
where
$$
\lvert\R(p)\rvert^2:=\R_{ijkl}(p)\R_{ijkl}(p)g^{ii}(p)g^{jj}(p)g^{kk}(p)g^{ll}(p),
$$
$$
\lvert\Ric(p)\rvert^2:=\Ric_{ij}(p)\Ric_{ij}(p)g^{ii}(p)g^{jj}(p).
$$
Above, $\R_{ijkl}(p)$ is the Riemann curvature tensor, $\Ric_{ij}(p)$ is the Ricci tensor, $\Scal(p)$ is the scalar curvature.

We use (\ref{term u_2}), to obtain

\begin{equation}\label{computation of c in 4 dim}
\begin{split}
c
=&-\frac{1}{720}(4\pi)^{-2}\int_N(12\Delta_N\Scal(x)+5(\Scal(x)-6)^2-2\lvert\Ric(x)\rvert^2\\
&+2\lvert\R(x)\rvert^2)\dvol_N.
\end{split}
\end{equation}

Since $(N,g_N)$ is closed, for any smooth function $\tau\in C^\infty(N)$ we have

\begin{align*}
\int_N\Delta_{N}\tau \dvol_N
=&\int_N(d^{\dag}d\tau)\dvol_N
=\int_N(d^{\dag}d\tau,1)_x \dvol_N\\
=&\int_N(d\tau,d1)_x \dvol_N
=\int_N(d\tau,0)_x \dvol_N=0,
\end{align*}
where $(\cdot,\cdot)_x$ denotes the inner product on the tangent space $T_xM$. Hence (\ref{computation of c in 4 dim}) becomes

\begin{align}\label{second computation of c in 4 dim}
c=&-\frac{1}{720}(4\pi)^{-2}\int_N(5(\Scal(x)-6)^2-2\lvert\Ric(x)\rvert^2+2\lvert\R(x)\rvert^2)\dvol_N.
\end{align}

For $\dim N=3$, the Ricci curvature tensor determines the Riemann curvature tensor, and we have the identity at every point $x\in N$
\begin{align}\label{curvature identity in 3 dim}
\lvert\R(x)\rvert^2=4\lvert\Ric(x)\rvert^2-\Scal^2(x).
\end{align}
We compute (\ref{second computation of c in 4 dim}) using (\ref{curvature identity in 3 dim}),
\begin{align*}
c
=&-\frac{1}{720}(4\pi)^{-2}\int_N\Big(5(\Scal(x)-6)^2-2\Scal^2(x)+6\lvert\Ric(x)\rvert^2\Big)\dvol_N\\
=&-\frac{1}{720}(4\pi)^{-2}\int_N\Big(5(\Scal(x)-6)^2-2\Scal^2(x)+24\Scal(x)-2\cdot36\\
&+6(\lvert\Ric(x)\rvert^2-4\Scal(x)+12\Big)\dvol_N\\
=&-\frac{1}{720}(4\pi)^{-2}\int_N\Big(3(\Scal(x)-6)^2+6\lvert\Ric_{ij}(x)-2g_{ij}\rvert^2\Big)\dvol_N\leq0.
\end{align*}
This expression is equal to zero if and only if $\Scal(x)\equiv6$ and $\Ric_{ij}(x)=2g_{ij}(x)$; equivalently if and only if the sectional curvature of $(N,g_N)$ is equal to one. Therefore $(N,g_N)$ is isometric to a spherical space form. We have proved the first part of the theorem.

(2)
If the logarithmic term is equal to zero, then by the first claim of the theorem, the neighbourhood of the singularity $(U,g_U)$ has the following metric
$$
g_U=dr^2+r^2g_{S^3},
$$
where $g_{S^3}$ is the round metric on the unit sphere. Thus $(U,g_U)$ is isometric to the punctured ball. The ball is a flat space, i.e. the curvature tensor vanishes for every $p\in U$, hence $u_j(p)=0$ for $j>0$ and $u_0(p)\equiv1$. Therefore,
\begin{equation*}
\tilde{a}_j
=\fint_Mu_j(p)\dvol_M
=
\begin{cases}
\Vol(M), \text{ for } j=0,\\
\int_Mu_j(p)\dvol_M, \text{ for } j>0.
\end{cases}
\end{equation*}
Therefore no regularization is needed for $\tilde{a}_j$, for any $j\geq0$, and the only term that has a contribution of the singularity is the constant term $b$.

\end{proof}

Assume that $(N,g_N)$ is isometric to a spherical space form and compute the constant term $b$ in the heat trace expansion on $(M,g)$. Since $c=0$, we have $\Res_1\zeta_N^1(-1/2)=0$. By Theorem~\ref{main theorem}~(b),
\begin{align}\label{computation of b for 4 dim spherical space form}
b=-\frac12\zeta_N^1(-1/2)-\frac14B_2\Res_1\zeta_N^1(1/2)-\frac18B_4\Res_N^1(3/2).
\end{align}
In the proof of Theorem~\ref{main theorem} we obtained,
$$
\Res_1\zeta^1_N(3/2-l)=\frac{1}{(4\pi)^{\frac32}\Gamma(\frac32-l)}\sum_{i=0}^l\frac{(-1)^i}{i!}a^N_{l-i}.
$$
Using $\Gamma(1/2)=\sqrt{\pi}$, $\Gamma(3/2)=\frac12\sqrt{\pi}$ and $\Scal(x)\equiv6$, we compute
\begin{align*}
\Res_1\zeta^1_N(1/2)
=&\frac{1}{(4\pi)^{\frac32}\Gamma(\frac12)}(a^N_1-a^N_0)\\
=&\frac{2}{(4\pi)^2}\left(\frac16\int_N\Scal(x)\dvol_N-\Vol(N)\right)=0,
\end{align*}
and
\begin{align*}
\Res_1\zeta^1_N(3/2)
=\frac{1}{(4\pi)^{\frac32}\Gamma(\frac32)}(a^N_0)
=\frac{1}{4\pi^2}\Vol(N).
\end{align*}
Then (\ref{computation of b for 4 dim spherical space form}) becomes
\begin{align*}
b=-\frac12\zeta_N^1(-1/2)+\frac{1}{960\pi^2}\Vol(N).
\end{align*}

In the four-dimensional case, the logarithmic term in the expansion vanishes precisely when the cross-section is a spherical space form, and we expect that the vanishing of a the constant term will imply that the cross-section is the sphere $(S^3,g_{\text{round}})$, but this is not yet clear beyond the case of lens spaces. The constant term $b$ was computed for lens spaces $(S^3/\mathbb{Z}^k,g_{\text{round}})$, see \cite[pp.\,2281--2282]{DB}, also \cite[(8.4)]{BKD}
\begin{equation}
b=\frac{(k^2+11)(k^2-1)}{720k}.
\end{equation}
We observe that $b=0$ if and only if $k=1$. However there are other types of spherical space forms for which the expression of the constant term is not clear yet.
\end{subsection}


\begin{subsubsection}{Example \texorpdfstring{$N=S^3_A$}{Lg}}\label{example S^3}
Consider a four-dimensional manifold $(M,g)$ with one conic singularity and the following metric near the singularity,
$$
g_{\text{conic}}=dr^2+r^2g_N.
$$
Above $(N,g_N)$ is the three-dimensional sphere $N=S^3_A$ with radius $A>0$. We compute the logarithmic term in the heat trace expansion on $M$ and show that it is equal to zero if and only if $A=1$.

The metric on the sphere is given by

$$
g_N=A^2(d\varphi_1^2+\sin^2\varphi_1d\varphi_2^2+\sin^2\varphi_1\sin^2\varphi_2d\varphi_3^2),
$$
where $0\leq\varphi_i<2\pi$ for $i=1,2,3$, and the volume $\Vol(S^3_A)=2\pi^2A^3$.

We compute the norm of the Ricci curvature tensor on $S^3_A$

$$
\lvert\Ric(x)\rvert^2=\left\vert2\frac{1}{A^2}g_{ij}(x)\right\vert^2=4\frac{1}{A^4}g_{ij}(x)g_{ij}(x)g^{ii}(x)g^{jj}(x)=\frac{12}{A^4},
$$
and the scalar curvature

$$
\Scal(x)=\frac{6}{A^2}.
$$

By (\ref{second computation of c in 4 dim}), the logarithmic term is
\begin{align*}
c
=&-\frac{1}{720}(4\pi)^{-2}\int_{S^3_A}\bigg[5(\Scal(x)-6)^2+6\lvert\Ric(x)\rvert^2-2\Scal(x)^2\bigg]\dvol_{S^3_A}\\
=&-\frac{1}{720}(4\pi)^{-2}\frac{5\cdot36(A^2-1)^2}{A^4}v\Vol(S^3_A)
=-\frac{(A^2-1)^2}{32A}.
\end{align*}

We see that the logarithmic term in the heat trace expansion is zero if and only if $A=1$. This example illustrates that for $N=S^3_A$ the logarithmic term in the heat trace expansion vanishes if and only if $A=1$. Moreover, by Lemma~\ref{lemma constant term for n sphere} for $N=S^3_1$ the constant term $b$ vanishes as well and the heat trace expansion on $M$ looks exactly like the heat trace expansion on a smooth compact manifold.
\end{subsubsection}

\begin{subsubsection}{Example \texorpdfstring{$N=T^3$}{Lg}}

Consider a four-dimensional manifold $(M,g)$ with one conic singularity and the following metric near the singularity
$$
g_{\text{conic}}=dr^2+r^2g_N,
$$
where $(N,g_N)$ is the three-dimensional flat torus $T^3=S^1\times S^1\times S^1$. We show that the logarithmic term in the heat trace expansion on $M$ is non-zero and try to compute the constant term $b$.

By (\ref{second computation of c in 4 dim}),

\begin{align*}
c
=-\frac{1}{720}(4\pi)^{-2}\int_{T^3}(5\cdot36)\dvol_{T^3}
=-\frac{1}{32\pi^2}\Vol(T^3)\neq0.
\end{align*}

Now we try to compute the constant term

\begin{align*}
b=&-\frac{1}{2}\Res_0\zeta^1_N(-1/2)
+\frac{\Gamma'(-\frac1 2)}{4\sqrt{\pi}}\Res_1\zeta^1_N(-1/2)\\
&-\frac1 4\sum_{j=1}^{m/2}j^{-1}B_{2j}\Res_1\zeta^1_N(j-1/2).
\end{align*}

By \cite[Prop.~B.I.2]{BGM}, the eigenvalues of the Laplace operator on the flat torus are
$$
\{4\pi^2|n|^2, \text{ where }n=(n_1,n_2,n_3)\in\mathbb{Z}^3\},
$$
therefore
\begin{align}\label{zeta on 3 torus}
\zeta^1_N(s)
=\sum_{\lambda_k\in Spec\Delta_N}(\lambda_k+1)^{-s}
=\sum_{n_1,n_2,n_3\in\mathbb{Z}}\bigg(4\pi^2(n_1^2+n_2^2+n_3^2)+1\bigg)^{-s}.
\end{align}

On the right hand side we have the Epstein function

$$
E_k(s;a_1,\dots,a_k,c)
:=\sum_{n_1,\dots,n_k\in\mathbb{Z}}\bigg(a_1n_1^2+\dots+a_kn_k^2+c\bigg)^{-s}.
$$

By \cite[A.28]{K} we get

\begin{align*}
\zeta^1_N(s)
=&E_3(s;4\pi^2,4\pi^2,4\pi^2,1)
=\frac{\pi^{3/2}}{(2\pi)^{3-2s}}\frac{\Gamma(s-3/2)}{\Gamma(s)}+\\
    &+\frac{2\pi^s}{(2\pi)^{3/2-s}\Gamma(s)}
    \sum_{n\in\mathbb{Z}^3\setminus\{0\}}\bigg(n_1^2+n_2^2+n_3^2\bigg)^{s/2-3/4}\times\\
    &\times K_{3/2-s}\left(\sqrt{n_1^2+n_2^2+n_3^2}\right).
\end{align*}

The poles are located at $s=3/2-j$ for $j\in\mathbb{N}_0$, these poles are coming from $\Gamma(s-3/2)$. To compute the constant term $b$, besides the residues of $\zeta_N^1(s)$, we need to compute the finite part $\Res_0\zeta_N^1(-1/2)$. The above computation shows that these are given by the infinite sums involving Riemann zeta function, therefore here we cannot find clear and nice expression for $\Res_0\zeta^1_N(-1/2)$, consequently we cannot find a nice expression for the constant term in the heat trace expansion in this case.

This example illustrates that if the logarithmic term in the heat trace expansion on $(M,g)$ is non-zero, then it is harder to find a nice expression for the constant term $b$.
\end{subsubsection}

\begin{subsection}{Criterion for the logarithmic term to vanish}

In this section we develop a geometric condition that is imposed on the cross-section manifold $(N,g_N)$ by the vanishing of the logarithmic term in the heat trace expansion on $(M,g)$.

Let $(N,g_N)$ be a closed manifold of dimension $n$. As before, we denote the coefficients in the heat trace expansion (\ref{smooth expansion}) on $(N,g_N)$ by $a^N_j, j\geq0$.

\begin{lemm}
Let $(M,g)$ be an even-dimensional manifold with a conic singularity. The logarithmic term in the heat trace expansion on $(M,g)$ is equal to zero if and only if the following equality holds for the heat trace coefficients of the cross-section manifold $(N,g_N)$
\begin{align}
a^N_{\frac{n+1}{2}}=\sum_{k=1}^{\frac{n+1}{2}}(-1)^{k+1}\frac{(n-1)^{2k}}{4^k k!}a^N_{\frac{n+1}{2}-k}.
\end{align}
\end{lemm}

\begin{proof}
By Theorem~1.1 (c), we have
\begin{align*}
c
=\frac{1}{2(4\pi)^{\frac{m}{2}}}\bigg(-a^N_{\frac{m}{2}}+\sum_{k=1}^{\frac{m}{2}}(-1)^{k+1}\frac{(m-2)^{2k}}{4^kk!}a^N_{\frac{m}{2}-k}\bigg).
\end{align*}
Use $m=n+1$, to obtain
\begin{align*}
c=\frac{1}{2(4\pi)^{\frac{n+1}{2}}}\bigg(-a^N_{\frac{n+1}{2}}+\sum_{k=1}^{\frac{n+1}{2}}(-1)^{k+1}\frac{(n-1)^{2k}}{4^kk!}a^N_{\frac{n+1}{2}-k}\bigg).
\end{align*}
If $n$ is odd, $c=0$ is equivalent to 
\begin{align*}
a^N_{\frac{n+1}{2}}=\sum_{k=1}^{\frac{n+1}{2}}(-1)^{k+1}\frac{(n-1)^{2i}}{4^kk!}a^N_{\frac{n+1}{2}-k}.
\end{align*}
\end{proof}

\end{subsection}

\end{section}

\begin{section}{Explicit expressions of the singular terms}

In this section we obtain explicit expressions for the logarithmic term and the constant term $b$ in the heat trace expansion on $(M,g)$ in Theorem~\ref{main theorem} for some particular $n$-dimensional cross-sections $(N,g_N)$.

\begin{subsection}{Spaces with constant sectional curvature}

Let $l\in\mathbb{N}$. Denote by $K_i^l$ the coefficients in the following polynomial
\begin{align}\label{poly}
\prod_{q=0}^{l-1}\big(v^2-q^2\big)=\sum_{i=1}^{l}K^l_iv^{2i}.
\end{align}

Let $(N,g_N)$ be a closed manifold and let $a^N_j, j\geq0$ be the coefficients in the heat trace expansion (\ref{smooth expansion}) on $(N,g_N)$.

\begin{theo}[{\cite[Theorem 1.3.1]{P}} ]\label{eigenvalues on n-sphere}
The coefficients in the heat trace expansion on the odd-dimensional unit sphere $(S^n,g_{\text{round}})$ are
\begin{align*}
a^{S^n}_j=(4\pi)^{\frac{n}{2}}\sum_{l=1}^{\frac{n-1}{2}}\frac{(\frac{n-1}{2})^{2j-n+1+2l}\Gamma(l+\frac12)K_l^{\frac{n-1}{2}}}{(j+l-\frac{n-1}{2})!(n-1)!},
\end{align*}
where numbers $K_l^{\frac{n-1}{2}}$ are given by (\ref{poly}).

In particular,
\begin{align*}
a^{S^5}_j=\frac{2^{2j-1}(6-j)\pi^3}{3j!},
\end{align*}
and
\begin{align*}
a^{S^7}_j=\frac{3^{2j-6}(16j^2-286j+1215)\pi^4}{5j!}.
\end{align*}
\end{theo}

Let $(M,g)$ be a manifold with a conic singularity and let $(N,g_N)$ be a cross-section near the singularity.

\begin{theo}\label{polynomial}
Let $(N,g_N)$ be an $n$-dimensional, $n$-odd, manifold with constant sectional curvature $\kappa$. Then the logarithmic term in the heat trace expansion on $(M,g)$ can be written as the following polynomial in $\kappa$ of degree $\frac{n+1}{2}$
$$
c=
\frac{1}{4\sqrt{\pi}}\frac{\Vol(N)}{\Vol(S^n)}\sum_{k=0}^{\frac{n+1}{2}}(-1)^{k+1}\frac{(n-1)^{2k}}{4^kk!}\sum_{l=1}^{\frac{n-1}{2}}\frac{(\frac{n-1}{2})^{2l-2k+2}\Gamma(l+\frac12)K_l^{\frac{n-1}{2}}}{(l-k+1)!(n-1)!}\kappa^{\frac{n+1}{2}-k},
$$
where numbers $K_l^{\frac{n-1}{2}}$ are given by (\ref{poly}).
\end{theo}

\begin{proof}
Since $(N,g)$ has constant sectional curvature $\kappa$, the Riemannian curvature tensor is
$$
\R_{ijkl}=\kappa(g_{ik}g_{jl}-g_{il}g_{jk}),
$$
the Ricci tensor is
$$
\Ric_{ij}=\kappa(n-1)g_{ij},
$$
the scalar curvature is
$$
\Scal=\kappa n(n-1),
$$
where $g_{ij}$ is the round metric on the unit sphere $S^n$.

Consider the heat trace expansion (\ref{smooth expansion}) on $(N,g_N)$.
Since the curvature is constant, every $a^N_j$ in the heat trace expansion is expressed as an integral over $(N,g_N)$ of a constant. By \cite[Lemma~E.IV.5]{BGM},  
\begin{align}\label{constant curvature terms}
a^N_j=\kappa^ja^{S^n}_j\frac{\Vol(N)}{\Vol(S^n)},
\end{align}
where $a^{S^n}_j$ is the coefficient in the heat trace expansion on $(S^n,g_{\text{round}})$ and $\Vol(N)$ is the volume of $(N,g_N)$.

By Theorem~1.1 (c), the logarithmic term in the heat trace expansion is
\begin{align*}
c
=&\frac{1}{2(4\pi)^{\frac{n+1}{2}}}\sum_{k=0}^{\frac{n+1}{2}}(-1)^{k+1}\frac{(n-1)^{2k}}{4^kk!}a^N_{\frac{n+1}{2}-k}.
\end{align*}
Use (\ref{constant curvature terms}), to obtain
\begin{align}\label{equation1}
c
=&\frac{1}{2(4\pi)^{\frac{n+1}{2}}}\sum_{k=0}^{\frac{n+1}{2}}(-1)^{k+1}\frac{(n-1)^{2k}}{4^kk!}\frac{\Vol(N)}{\Vol(S^n)}a^{S^n}_{\frac{n+1}{2}-k}\kappa^{\frac{n+1}{2}-k}.
\end{align}
By Theorem \ref{eigenvalues on n-sphere},
\begin{align*}
a^{S^n}_j=(4\pi)^{\frac{n}{2}}\sum_{l=1}^{\frac{n-1}{2}}\frac{(\frac{n-1}{2})^{2j-n+1+2l}\Gamma(l+\frac12)K_l^{\frac{n-1}{2}}}{(j+l-\frac{n-1}{2})!(n-1)!},
\end{align*}
hence
\begin{align}\label{equation2}
a^{S^n}_{\frac{n+1}{2}-k}=(4\pi)^{\frac{n}{2}}\sum_{l=1}^{\frac{n-1}{2}}\frac{(\frac{n-1}{2})^{2l-2k+2}\Gamma(l+\frac12)K_l^{\frac{n-1}{2}}}{(l-k+1)!(n-1)!}.
\end{align}
It remains to put (\ref{equation2}) into (\ref{equation1}) to obtain the desired expression for the logarithmic term
\begin{align*}
c
=&\frac{1}{4\sqrt{\pi}}\frac{\Vol(N)}{\Vol(S^n)}\sum_{k=0}^{\frac{n+1}{2}}(-1)^{k+1}\frac{(n-1)^{2k}}{4^kk!}\sum_{l=1}^{\frac{n-1}{2}}\frac{(\frac{n-1}{2})^{2l-2k+2}\Gamma(l+\frac12)K_l^{\frac{n-1}{2}}}{(l-k+1)!(n-1)!}\kappa^{\frac{n+1}{2}-k}.
\end{align*}
\end{proof}

The coefficients in the polynomial in Theorem \ref{polynomial} might seem complicated. In order to better understand which geometrical information we obtain from Theorem \ref{polynomial}, we consider some low-dimensional cases.


\begin{coro}
Let $(N,g_N)$ be a three-dimensional manifold with constant sectional curvature $\kappa$. The logarithmic term in the heat trace expansion on $(M,g)$ is zero if and only if $\kappa=1$.
\end{coro}
\begin{proof}
By \cite[Section 9]{CW}, the coefficients in the heat trace expansion on $S^3$ are
\begin{align*}
a^{S^3}_j=\frac{2\pi^2}{j!}.
\end{align*}
We compute the first three coefficients
\begin{align}\label{sphere 3 dim}
a^{S^3}_0=2\pi^2, &&
a^{S^3}_1=2\pi^2, &&
a^{S^3}_2=\pi^2.
\end{align}
By (\ref{equation1}),
\begin{align*}
c
=&\frac{1}{2(4\pi)^2}\frac{\Vol(N)}{\Vol(S^3)}\sum_{k=0}^2(-1)^{k+1}\frac{2^{2k}}{4^kk!}a^{S^3}_{2-k}\kappa^{2-k}\\
=&-\frac{1}{2(4\pi)^2}\frac{\Vol(N)}{\Vol(S^3)}\left(a^{S^3}_2\kappa^2-a^{S^3}_1\kappa+\frac12a^{S^3}_0\right).
\end{align*}
Using (\ref{sphere 3 dim}), we obtain
\begin{align*}
c
=&-\frac{1}{2(4\pi)^2}\frac{\Vol(N)}{\Vol(S^3)}\pi^2(\kappa^2-2\kappa+1).
\end{align*}
Finally, the logarithmic term $c=0$ if and only if $\kappa=1$.
\end{proof}


\begin{coro}
Let $(N,g_N)$ be a five-dimensional manifold with constant sectional curvature $\kappa$. The logarithmic term in the heat trace expansion on $(M,g)$ is zero if and only if one of the following is true $\kappa=1$ or $\kappa=2$.
\end{coro}
\begin{proof}
We compute the logarithmic term in the heat trace expansion. By Theorem \ref{eigenvalues on n-sphere}, we compute
\begin{align}\label{sphere 5 dim}
a^{S^5}_0=\pi^3, &&
a^{S^5}_1=\frac{10\pi^3}{3}, &&
a^{S^5}_2=\frac{16\pi^3}{3}, &&
a^{S^5}_3=\frac{16\pi^3}{3}.
\end{align}
By (\ref{equation1}),
\begin{align*}
c
=&\frac{1}{2(4\pi)^3}\frac{\Vol(N)}{\Vol(S^5)}\sum_{k=0}^3(-1)^{k+1}\frac{4^{2k}}{4^kk!}a^{S^3}_{3-k}\kappa^{3-k}\\
=&-\frac{1}{2(4\pi)^3}\frac{\Vol(N)}{\Vol(S^5)}\left(a^{S^3}_{3}\kappa^{3}-4a^{S^3}_{2}\kappa^{2}+8a^{S^3}_{1}\kappa-\frac{32}{3}a^{S^3}_0\right).
\end{align*}
Using (\ref{sphere 5 dim}), we obtain
\begin{align*}
c
=&-\frac{1}{2(4\pi)^3}\frac{\Vol(N)}{\Vol(S^5)}\frac{16\pi^3}{3}\left(\kappa^{3}-4\kappa^{2}+5\kappa-2\right).
\end{align*}
Since $\kappa^{3}-4\kappa^{2}+5\kappa-2=(\kappa-1)^2(\kappa-2)$,
the logarithmic term $c=0$ if and only if one of the following is true $\kappa=1$ or $\kappa=2$.
\end{proof}


Finally, we apply Theorem \ref{polynomial} in the seven-dimensional case.

\begin{coro}
Let $(N,g_N)$ be a seven-dimensional manifold with constant sectional curvature $\kappa$. The logarithmic term in the heat trace expansion on $(M,g)$ is zero if and only if one of the following is true $\kappa=1$ or $\kappa=\frac{225}{109}\pm\frac{36\sqrt{5}}{109}$.
\end{coro}
\begin{proof}
We compute the logarithmic term in the heat trace expansion. By Theorem \ref{eigenvalues on n-sphere}, we compute
\begin{align*}
& a^{S^7}_0=\frac{5}{3}\frac{\pi^4}{5}, 
&& a^{S^7}_1=\frac{35}{3}\frac{\pi^4}{5}, 
& a^{S^7}_2=\frac{707}{18}\frac{\pi^4}{5},\\
& a^{S^7}_3=\frac{167}{2}\frac{\pi^4}{5},
&& a^{S^7}_4=\frac{3\cdot327}{8}\frac{\pi^4}{5}.
\end{align*}
By (\ref{equation1}),
\begin{align*}
c
=&\frac{1}{2(4\pi)^4}\frac{\Vol(N)}{\Vol(S^7)}\sum_{k=0}^4(-1)^{k+1}\frac{3^{2k}}{k!}a^{S^7}_{4-k}\kappa^{4-k}\\
=&-\frac{1}{2(4\pi)^4}\frac{\Vol(N)}{\Vol(S^7)}\left(a^{S^7}_4\kappa^4-9A^{S^7}_3\kappa^3+\frac{3^4}{2}a^{S^7}_2\kappa^2-\frac{3^5}{2}a^{S^7}_1\kappa+\frac{3^7}{8}a^{S^7}_0\right).
\end{align*}
We use the precise expressions for the heat trace expansion coefficients on $S^7$, to obtain
\begin{align*}
c
=&-\frac{1}{2(4\pi)^4}\frac{\Vol(N)}{\Vol(S^7)}\frac{\pi^4}{5}\left(\frac{3\cdot327}{8}\kappa^4-\frac{9\cdot167}{2}\kappa^3+\frac{3^2\cdot707}{4}\kappa^2\right.\\
&\left.-\frac{3^4\cdot35}{2}\kappa+\frac{3^6\cdot5}{8}\right)\\
=&-\frac{1}{2(4\pi)^4}\frac{\Vol(N)}{\Vol(S^7)}\frac{\pi^4}{5}\frac{9}{8}\left(109\kappa^4-668\kappa^3+1414\kappa^2-1260\kappa+405\right).
\end{align*}
Since $109\kappa^4-668\kappa^3+1414\kappa^2-1260\kappa+405=(\kappa-1)^2(109\kappa^2-450\kappa+405)$, the logarithmic term $c$ is zero if and only if one of the following is true $\kappa=1$ or $\kappa=\frac{225}{109}\pm\frac{36\sqrt{5}}{109}$.
\end{proof}

Summing up, in this section we consider an $n+1$-dimensional manifold $(M,g)$ with a conic singularity such that the cross-section $(N,g_N)$ near the singularity has constant sectional curvature, and assume that the logarithmic term in the heat trace expansion on $(M,g)$ vanishes. While in the four-dimensional case the only possible cross-section $(N,g_N)$ is a spherical space form, in the six-dimensional case besides a spherical space form we may have a cross-section with sectional curvature $\kappa=2$. Furthermore, in the eight-dimensional case besides spherical space forms we may have cross-sections with even more peculiar sectional curvatures $\kappa=\frac{225}{109}\pm\frac{36\sqrt{5}}{109}$.

\end{subsection}

\begin{subsection}{The logarithmic and the constant term for \texorpdfstring{$N=S^n$}{Lg}}

In this section we show that the logarithmic term $c$ and the constant term $b$ in the heat trace expansion on $(M,g)$ for $N=S^n$ with the round metric, both are equal to zero.

\begin{prop}[{\cite[Prop.C.I.1, Cor.C.I.3, p.160]{BGM}} ]\label{eigenvalues on sphere}
The sphere $(S^n,g_{\text{round}})$ has eigenvalues 
$$
\lambda_k=k(n+k-1), \;\;\;\; k\geq0,
$$
with multiplicities
$$
\mu_k=\frac{(n+k-2)(n+k-3)\dots(n+1)n}{k!}(n+2k-1).
$$
\end{prop}

First, we prove a more general result. We consider spherical space forms.

\begin{lemm}\label{log term for spherical space forms}
Let $(M,g)$ be a manifold with conic singularities and let the cross-section manifold near every singularity be isometric to some spherical space form. Then the logarithmic term in the heat trace expansion on $(M,g)$ is equal to zero. 
\end{lemm}

\begin{proof}
We consider one cross-section $(N,g_N)$, which is isometric to a spherical space form, and prove that its contribution to the logarithmic term in the heat trace expansion on $(M,g)$ is zero. If every cross-section is isometric to a spherical space form, we conclude that the sum of all contributions is zero as desired.

Assume that $(M,g)$ is odd-dimensional, by Theorem~1.1~(c) we obtain $c=0$.

Assume now that $(M,g)$ is even-dimensional. By (\ref{constant curvature terms}), the coefficients $a_j^N$ in the heat trace expansion on $(N,g_N)$ are proportional to the coefficients in the heat trace expansion on $S^n$
\begin{align}\label{terms for spherical space form}
a^N_j=a^{S^n}_j\frac{\Vol(N)}{\Vol(S^n)}.
\end{align}
We use Theorem~1.1~(c) and (\ref{terms for spherical space form}), to obtain
\begin{align}
c
=&\frac{1}{2(4\pi)^{\frac{n+1}{2}}}\frac{\Vol(N)}{\Vol(S^n)}\sum_{k=0}^{\frac{n+1}{2}}(-1)^{k+1}\frac{(n-1)^{2k}}{4^kk!}a^{S^n}_{\frac{n+1}{2}-k}.
\end{align}
Therefore the logarithmic term $c$ for $(N,g_{\text{round}})$ differs from the logarithmic term $c_{S^n}$ for $(S^n,g_{\text{round}})$ by the factor $\frac{\Vol(N)}{\Vol(S^n)}$. Assume that we prove that $c_{S^n}=0$, then it follows that the logarithmic term $c$ is zero for any spherical space form.

Now we prove that the logarithmic term $c_{S^n}$ in the heat trace expansion on $(M,g)$ for $N=S^n$ is zero.

Observe,
\begin{align}\label{equation4}
\frac{(n+k-2)(n+k-3)\dots(n+1)n}{k!}=\frac{(n+k-2)(n+k-3)\dots(k+1)}{(n-1)!}.
\end{align}
By Proposition \ref{eigenvalues on sphere} and (\ref{equation4}), we obtain

\begin{align*}
\zeta^{\frac{n-1}{2}}_{S^n}(s)\\
=&\sum_{k\geq0}\frac{(n+k-2)(n+k-3)\dots(k+1)}{(n-1)!}(n+2k-1)\left(k+\frac{n-1}{2}\right)^{-2s}\\
=&\sum_{k\geq0}\frac{2(n+k-2)(n+k-3)\dots(k+1)}{(n-1)!}\left(k+\frac{n-1}{2}\right)^{-2s+1}.
\end{align*}
Let $n=2u+1$ and $k+u=l$, then

\begin{align*}
\zeta^{\frac{n-1}{2}}_{S^n}(s)
=&\sum_{k\geq0}\frac{2(2u+k-1)(2u+k-2)\dots(k+1)}{(2u)!}\left(k+u\right)^{-2s+1}\\
=&\sum_{l\geq u}\frac{2(l+u-1)(l+u-2)\dots(l-u+1)}{(2u)!}l^{-2s+1}\\
=&\sum_{l\geq u}\frac{2\prod_{q=1}^{u-1}(l^2-q^2)}{(2u)!}l^{-2s+2}
=\frac{2}{(2u)!}\sum_{l\geq 1}l^{-2s+2}\prod_{q=1}^{u-1}(l^2-q^2).
\end{align*}
The last equality is true, because $\sum_{l=1}^{u-1}\prod_{q=1}^{u-1}(l^2-q^2)=0$.

Let $K^u_i$ be defined by (\ref{poly}). Then
\begin{align*}
\zeta^{\frac{n-1}{2}}_{S^n}(s)
=\frac{2}{(2u)!}\sum_{l\geq 1}l^{-2s}\sum_{i=1}^{u}K^u_il^{2i}\\
=\frac{2}{(2u)!}\sum_{l\geq 1}\sum_{i=1}^{u}K^u_il^{-2s+2i}.
\end{align*}
Finally
\begin{align}\label{zeta function on n dim sphere}
\zeta^{\frac{n-1}{2}}_{S^n}(s)
=&\frac{2}{(2u)!}\sum_{i=1}^{u}K^u_i\zeta(2s-2i),
\end{align}
where $\zeta(s)$ is the Riemann zeta function. We compute

\begin{align*}
\Res_1\zeta^{\frac{n-1}{2}}_{S^n}(-1/2)
=&\frac{2}{(2u)!}\sum_{i=1}^{u}K^u_i\Res_1\zeta(-1-2i)=0,
\end{align*}
because the Riemann zeta function has no poles at points $-1-2i$ for $i=1,\dots,u$. We conclude that in the case of spherical space form $(N,g_{\text{round}})$, the logarithmic term in the heat trace expansion on $(M,g)$ is zero
$$
c=\frac12\frac{\Vol(N)}{\Vol(S^n)}Res_1\zeta^{\frac{n-1}{2}}_{S^n}(-1/2)=0.
$$
\end{proof}

\begin{lemm}\label{lemma constant term for n sphere}
Let $(M,g)$ be a manifold with conic singularities and let the cross-section manifold near every singularity be isometric to the odd-dimensional unit sphere $(S^n,g_{\text{round}})$. Then the constant term $b$ in the heat trace expansion on $(M,g)$ is equal to zero.
\end{lemm}
\begin{proof}
We consider one cross-section $(N,g_{\text{round}})$ and prove that its contribution to the constant term $b$ in the heat trace expansion on $(M,g)$ is zero. If every cross-section is isometric to $(S^n,g_{\text{round}})$, we conclude that the sum of all contributions is zero as desired.

By (\ref{zeta function on n dim sphere}),

\begin{align*}
\zeta^{\frac{n-1}{2}}_{S^n}(s)
=&\frac{2}{(n-1)!}\sum_{i=1}^{\frac{n-1}{2}}K^{\frac{n-1}{2}}_i\zeta(2s-2i),
\end{align*}
using Theorem~1.1 (b), we obtain

\begin{equation*}
\begin{split}
b
=&-\Res_0\frac{1}{(n-1)!}\sum_{i=1}^{\frac{n-1}{2}}K^{\frac{n-1}{2}}_i\zeta(-1-2i)\\
&-\frac1 2\sum^{\frac{n+1}{2}}_{j=1}j^{-1}B_{2j}\frac{1}{(n-1)!}\sum_{i=1}^{\frac{n-1}{2}}K^{\frac{n-1}{2}}_i\Res_1\zeta(2j-1-2i).
\end{split}
\end{equation*}
Since $\Res_1\zeta(2j-1-2i)$ is non-zero only for $2j-1-2i=1$, we get the relation $j=i+1$. Using 
$$
\zeta(-d)=(-1)^d\frac{B_{d+1}}{d+1}, \;\;\;\;\;\;\;d\in\mathbb{N},
$$
we get

\begin{equation*}
\begin{split}
b
=&\frac{1}{(n-1)!}\sum_{i=1}^{\frac{n-1}{2}}K^{\frac{n-1}{2}}_i\frac{B_{2i+2}}{2i+2}-\frac1 2\sum_{i=1}^{\frac{n-1}{2}}(i+1)^{-1}B_{2i+2}\frac{1}{(n-1)!}K^{\frac{n-1}{2}}_i 
= 0,
\end{split}
\end{equation*}
as desired.
\end{proof}

\end{subsection}

\begin{subsection}{The logarithmic and the constant terms for \texorpdfstring{$N=\mathbb{R}P^n$}{Lg}}

In this section we consider a manifold $(M,g)$ with a conic singularity with the cross-section $N=\mathbb{R}P^n$ with the round metric. We show that in this case the logarithmic term in the heat trace expansion on $(M,g)$ is equal to zero. We compute the constant term $b$ and show that it is not equal to zero.

\begin{prop}[{\cite[Prop.C.II.1, p.166]{BGM}} ]\label{eigenvalues on projective space}
Let $(\mathbb{R}P^n,g_{\text{round}})$ be the projective space. The Laplace-Beltrami operator has the following eigenvalues
$$
\lambda_k=2k(n+2k-1), \;\;\;\; k\geq0,
$$
with multiplicities
$$
\frac{(2k+n-2)!}{(2k)!(n-1)!}(n+4k-1).
$$
\end{prop}

\begin{theo}\label{theorem constant term for projective space}
Let $(M,g)$ be a manifold with a conic singularity with the cross-section $N=\mathbb{R}P^n$, where $n=4v+1$, $v\in\mathbb{N}$. Then the logarithmic term in the heat trace expansion is zero and the constant term $b$ is not equal to zero
$$
b=\sum_{i=1}^{\frac{n-1}{2}}\left(2^{2i+2}-1\right)
\frac{1}{4(i+1)(n-1)!}B_{2i+2}K^{\frac{n-1}{2}}_{i}\neq0.
$$
Above $B_{j}$ are Bernoulli numbers and $K^n_j\in\mathbb{R}$ are given by (\ref{poly}).
\end{theo}

\begin{proof}
Since the projective space is a spherical space form, the first claim of the theorem follows from Lemma \ref{log term for spherical space forms}.

By Proposition~\ref{eigenvalues on projective space}, the multiplicity of the $k$-th eigenvalue of $\mathbb{R}P^n$ is
\begin{align*}
\mu_k
=&\frac{(2k+n-2)!}{(2k)!(n-1)!}(n+4k-1)\\
=&\frac{(2k+n-2)(2k+n-3)\dots(2k+1)}{(n-1)!}(4k+n-1).
\end{align*}
Therefore

\begin{align*}
\zeta^{\frac{n-1}{2}}_{\mathbb{R}P^n}(s)
=&\sum_{k\geq0}\frac{(2k+n-2)(2k+n-3)\dots(2k+1)}{(n-1)!}(4k+n-1)\times\\
\\&\times\left(2k+\frac{n-1}{2}\right)^{-2s}.
\end{align*}

Let $n=4v+1$ and $k+v=l$. Then

\begin{align*}
\zeta^{\frac{n-1}{2}}_{\mathbb{R}P^n}(s)
=&\sum_{k\geq0}\frac{(2k+4v-1)(2k+4v-2)\dots(2k+1)}{(4v)!}2\left(2k+2v\right)^{-2s+1}\\
=&\sum_{l\geq v}\frac{(2l+2v-1)(2l+2v-2)\dots(2l-2v+1)}{(4v)!}2\left(2l\right)^{-2s+1}\\
=&\sum_{l\geq v}\frac{\big((2l)^2-(2v-1)^2\big)\dots\big((2l)^2-1)\cdot 2l}{(4v)!}2^{-2s+2}l^{-2s+1}\\
=&\sum_{l\geq v}\frac{\prod_{q=1}^{2v-1}\big((2l)^2-q^2\big)}{(4v)!}
2^{-2s+3}l^{-2s+2}.
\end{align*}

The expression $\sum_{l=1}^{v-1}\prod_{q=1}^{2v-1}\big((2l)^2-q^2\big)=0$, because in every summand one of the terms $(2l)^2-q^2$ in the product is equal to zero. Therefore, we can add this to $\zeta^{\frac{n-1}{2}}_{\mathbb{R}P^n}(s)$ and it remains unchanged. We obtain

\begin{align*}
\zeta^{\frac{n-1}{2}}_{\mathbb{R}P^n}(s)
=&\sum_{l\geq 1}\frac{\prod_{q=0}^{2v-1}\big((2l)^2-q^2\big)}{(4v)!}
2^{-2s+1}l^{-2s}\\
=&\sum_{l\geq1}\sum_{i=1}^{2v}K^{2v}_i(2l)^{2i}\frac{2^{-2s+1}}{(4v)!}l^{-2s}\\
=&\sum_{l\geq1}\sum_{i=1}^{2v}\frac{2^{-2s+2i+1}}{(4v)!}K^{2v}_il^{-2s+2i}\\
=&\sum_{i=1}^{2v}\frac{2^{-2s+2i+1}}{(4v)!}K^{2v}_i\zeta(2s-2i).
\end{align*}
Above $\zeta(s)$ is the Riemann zeta function, which has only one simple pole at $s=1$. In particular, $\zeta(s)$ has no pole at $-2i-1$ for $i=1,\dots,2v$. Consequently, $\zeta^{\frac{n-1}{2}}_{\mathbb{R}P^n}(s)$ is analytic at $s=-1/2$. By
$$
\zeta(-d)=(-1)^d\frac{B_{d+1}}{d+1},\;\;\;\;\;\;\; d\in\mathbb{N},
$$
we obtain

\begin{align*}
\zeta^{\frac{n-1}{2}}_{\mathbb{R}P^n}(-1/2)
=&\sum_{i=1}^{2v}\frac{2^{2i+2}}{(4v)!}K^{2v}_i\zeta(-2i-1)\\
=&\sum_{i=1}^{2v}\frac{2^{2i+2}}{(4v)!}K^{2v}_i(-1)^{2i+1}\frac{B_{2i+2}}{2i+2},
\end{align*}
and simplify
\begin{align}\label{Casimir for projective space}
\zeta^{\frac{n-1}{2}}_{\mathbb{R}P^n}(-1/2)
=&-\sum_{i=1}^{2v}\frac{2^{2i+1}}{(i+1)(4v)!}B_{2i+2}K^{2v}_i.
\end{align}

To find the constant term $b$, it remains to compute the following sum of the residues

\begin{align*}
&\sum_{j=1}^{2v+1}j^{-1}B_{2j}\Res_1\zeta^{\frac{n-1}{2}}_{\mathbb{R}P^n}(j-1/2)\\
=&\sum_{j=1}^{2v+1}j^{-1}B_{2j}\sum_{i=1}^{2v}\frac{2^{-2j+2i+2}}{(4v)!}K^{2v}_i\Res_1\zeta(2j-2i-1).
\end{align*}
The summand in the above expression is not equal to zero only for $2j-2i-1=1$, equivalently $j=i+1$. We obtain

\begin{align}\label{residues for projective space}
\sum_{j=2}^{2v+1}j^{-1}B_{2j}\Res_1\zeta^{\frac{n-1}{2}}_{\mathbb{R}P^n}(j-1/2)
=&\sum_{i=1}^{2v}\frac{1}{(i+1)(4v)!}B_{2i+2}K^{2v}_{i}.
\end{align}

Using (\ref{Casimir for projective space}) and (\ref{residues for projective space}), we compute

\begin{align*}
b
=&-\frac1 2 \zeta^{\frac{n-1}{2}}_{\mathbb{R}P^n}(-1/2)-\frac1 4\sum_{j=1}^{2v+1}j^{-1}B_{2j}\Res_1\zeta^{\frac{n-1}{2}}_{\mathbb{R}P^n}(j-1/2)\\
=&\frac1 2\sum_{i=1}^{2v}\frac{2^{2i+1}}{(i+1)(4v)!}B_{2i+2}K^{2v}_i
-\frac1 4\sum_{i=1}^{2v}\frac{1}{(i+1)(4v)!}B_{2i+2}K^{2v}_{i}\\
=&\sum_{i=1}^{2v}\left(2^{2i}-\frac14\right)
\frac{1}{(i+1)(4v)!}B_{2i+2}K^{2v}_{i}\\
=&\sum_{i=1}^{\frac{n-1}{2}}\left(2^{2i+2}-1\right)
\frac{1}{4(i+1)(n-1)!}B_{2i+2}K^{\frac{n-1}{2}}_{i}.
\end{align*}

Note that by construction the numbers $K^{\frac{n-1}{2}}_i$ are alternating as $i$ grows, the Bernoulli numbers $B_{2i+2}$ are also alternating. We conclude that all products $B_{2i+2}E^{\frac{n-1}{2}}_{i}$ have the same sign and $b\neq0$.
\end{proof}

\end{subsection}

\begin{subsection}{The logarithmic term for \texorpdfstring{$N=T^n$}{Lg}}

In this section we consider an even-dimensional manifold $(M,g)$ with a conic singularity with the crosssection $N=T^n=S^1\times\dots\times S^1$ with the flat metric. We compute the logarithmic term in the heat trace expansion on $M$ and show that it is not equal to zero.

\begin{prop}[{\cite[Prop.B.I.2, p.148]{BGM}}]
Let $(T^n,g_{\text{flat}})$ be an $n$-dimensional torus with the flat metric. The Laplace-Beltrami operator has the following eigenvalues on $T^n$ 
$$
\lambda=\sum_{i=1}^n4\pi^2k_i^2, \;\;\;\; k=(k_1,\dots,k_n)\in\mathbb{Z}^n,
$$
the multiplicity of each eigenvalue is equal to one.
\end{prop}

\begin{theo}
Let $(M,g)$ be a manifold with a conic singularity with the cross-section $N=T^n$, where $n\geq3$ is odd. Then the logarithmic term in the heat trace expansion on $(M,g)$ is
$$
c=\frac{(-1)^{\frac{n+1}{2}}\left(n-1\right)^{n+1}}{2^{2n+3}\pi^{n+1/2}(\frac{n+1}{2})!}\neq0.
$$
\end{theo}

\begin{proof}
The shifted zeta function for the flat torus is
\begin{align*}
\zeta^{\frac{n-1}{2}}_{T^n}(s)
=&\sum_{k\in\mathbb{Z}^n}\left(4\pi^2k_1^2+\dots 4\pi^2k_n^2+\left(\frac{n-1}{2}\right)^2\right)^{-s},
\end{align*}
which is the Epstein zeta function. Since $n$ is odd, by \cite[pp.~319--320]{K}, $\zeta^{\frac{n-1}{2}}_{T^n}(s)$ is a meromorphic function with the simple poles at the points $s=\frac{n}{2},\dots,1,-\frac1 2,-\frac{2l+1}{2}, l\in\mathbb{N}$ with the residues
$$
Res_1\zeta^{\frac{n-1}{2}}_{T^n}(s)=\frac{(-1)^{n/2+s}\pi^{s/2}\left(\frac{n-1}{2}\right)^{n-2s}}{\sqrt{(4\pi^2)^n}\Gamma(s)\Gamma(\frac{n}{2}-s+1)}.
$$
Compute the residue at $s=-1/2$
\begin{align*}
Res_1\zeta^{\frac{n-1}{2}}_{T^n}(-1/2)
=&\frac{(-1)^{n/2-1/2}\pi^{-1/4}\left(\frac{n-1}{2}\right)^{n+1}}{\sqrt{(4\pi^2)^n}\Gamma(-\frac{1}{2})\Gamma(\frac{n}{2}+\frac12+1)}
=\frac{(-1)^{\frac{n-1}{2}}\pi^{-1/4}\left(\frac{n-1}{2}\right)^{n+1}}{(2\pi)^n(-2\sqrt{\pi})\Gamma(\frac{n+3}{2})}\\
=&\frac{(-1)^{\frac{n+1}{2}}\left(n-1\right)^{n+1}}{2^{n+2}(2\pi)^n\pi^{1/2}(\frac{n+1}{2})!}
=\frac{(-1)^{\frac{n+1}{2}}\left(n-1\right)^{n+1}}{2^{2n+2}\pi^{n+1/2}(\frac{n+1}{2})!}.
\end{align*}
The expression above is non-zero for $n\geq3$. By (\ref{logarithmic term}), the logarithmic term in the heat trace expansion is

\begin{align*}
c
=\frac{(-1)^{\frac{n+1}{2}}\left(n-1\right)^{n+1}}{2^{2n+3}\pi^{n+1/2}(\frac{n+1}{2})!}.
\end{align*}

\end{proof}

\end{subsection}
\end{section}

\end{document}